\documentclass[reqno]{amsart}

\usepackage{amsrefs}
\usepackage{calrsfs}
\usepackage{graphicx}
\usepackage{pstricks, enumerate}
\usepackage[pdfborder={0 0 0},colorlinks=true,pdfborderstyle={},bookmarksopen=false,linkcolor=black,citecolor=black,urlcolor=blue]{hyperref}

\newtheorem{theorem}{Theorem}
\newtheorem*{theorem*}{Theorem}
\newtheorem{lemma}[theorem]{Lemma}

\theoremstyle{remark}

\newtheorem*{remark*}{Remark}

\theoremstyle{definition}

\newenvironment{proofof}[1]{\noindent{\bf Proof of #1.}\hspace*{1em}}{\qed\bigskip}

\newcommand{\intav}[1]{\mathchoice {\mathop{\vrule width 6pt height 3 pt depth  -2.5pt
\kern -8pt \intop}\nolimits_{\kern -6pt#1}} {\mathop{\vrule width
5pt height 3  pt depth -2.6pt \kern -6pt \intop}\nolimits_{#1}}
{\mathop{\vrule width 5pt height 3 pt depth -2.6pt \kern -6pt
\intop}\nolimits_{#1}} {\mathop{\vrule width 5pt height 3 pt depth
-2.6pt \kern -6pt \intop}\nolimits_{#1}}}

\newcommand{\intavl}[1]{\mathchoice {\mathop{\vrule width 6pt height 3 pt depth  -2.5pt
\kern -8pt \intop}\limits_{\kern -6pt#1}} {\mathop{\vrule width 5pt
height 3  pt depth -2.6pt \kern -6pt \intop}\nolimits_{#1}}
{\mathop{\vrule width 5pt height 3 pt depth -2.6pt \kern -6pt
\intop}\nolimits_{#1}} {\mathop{\vrule width 5pt height 3 pt depth
-2.6pt \kern -6pt \intop}\nolimits_{#1}}}

\newcommand{\N}{\mathbb{N}}

\newcommand{\Z}{\mathbb{Z}}

\begin{document}

\title[Recurrence of VRRW on $\mathbb Z$ with sublinear weight]{vertex-reinforced random walk on $\mathbb Z$ with sub-square-root weights is recurrent.}

\author[Jun Chen]{Jun Chen}

\address{Division of the Humanities and Social Sciences, California Institute of Technology, Pasadena, CA 91125.}
\email{chenjun851009@gmail.com}

\author[Gady Kozma]{Gady Kozma}

\address{Faculty of Mathematics and Computer Science, Weizmann Institute of Science,  POB 26, 76100, Rehovot, Israel.}
\email{gady.kozma@weizmann.ac.il}

\subjclass[2010]{Primary: 60K35.}

\date{January 6, 2014}
\keywords{Vertex-reinforced random walk, recurrent vs transient, sublinear reinforcement, martingale}

\begin{abstract}
We prove that vertex-reinforced random walk on $\mathbb Z$
with weight of order $k^\alpha$, for $\alpha\in [0, 1/2)$, is
  recurrent. This confirms a conjecture of Volkov for $\alpha<1/2$. The conjecture for $\alpha\in [1/2, 1)$ remains open.
\end{abstract}

\maketitle

\section{Introduction}

Linearly vertex-reinforced random walk (VRRW for short), introduced by Pemantle \cite{pemantle1992vertex}, was studied
on $\mathbb Z$ by Pemantle and Volkov \cite{pemantle1999vertex}. A striking phenomenon was proved for this model
\cites{pemantle1999vertex, tarres2004vertex}: the random walk will eventually visit just
5 sites on $\mathbb Z$ almost surely.

In contrast, Volkov later in \cite{volkov2006phase} studied non-linear vertex reinforced random walk on $\mathbb Z$ with some
weight function $w:\{0,1,2,\dotsc\}\to(0,\infty)$. This process, denoted by
$(X_n, n\ge 0)$ is defined as follows. Fix $X_0=0$. Then for all $n\ge 0$,
\begin{equation}  \label{def_RW}
\mathbb P(X_{n+1}=X_n \pm 1 |X_1,\dotsc,X_n)=\frac{w(Z_n(X_n\pm 1))}{w(Z_n(X_n- 1))+w(Z_n(X_n+1))},
\end{equation}
where 
$Z_n(y)= \# \{m\le n: X_m=y\}$
is the local time in $y\in \mathbb Z$ at time $n$. For $w_k =
k^\alpha(c+o(1))$, $\alpha\ge 0$, Volkov proved the existence of phase
transition for this model. That is, there is a large time $T_0$ such that after $T_0$, the walk visits $2, 5$
or $\infty$ sites when $\alpha>1$, $\alpha=1$ and $\alpha<1$ respectively. In the case of $\alpha<1$, though it was
proved that the random walk will visit infinitely many sites, it is not clear whether it will visit every site
of $\Z$ infinitely many times with probability 1. Namely, the question whether the random walk is recurrent\footnote{The definition 
of recurrence we use here is that the random walk visits every vertex of $\mathbb Z$ infinitely many times almost surely and the definition
of transience is that the random walk visits every vertex of $\mathbb Z$ finitely many times almost surely.} was left open.

Recently, Schapira was able to move one step further towards a positive answer to this question, and in \cite{schapira20110}
proved a 0-1 law for VRRW on $\mathbb Z$ with weight of order $k^\alpha$, for $\alpha\in [0, 1/2)$.
In this paper, we show that in this regime the walk is in fact recurrent.

\begin{theorem*}  \label{main_result}
Vertex reinforced random walk on $\mathbb Z$ with weight $w(k)\approx k^\alpha$, $\alpha\in [0, 1/2)$, and montone increasing is recurrent.
\end{theorem*}

The notation $w(k)\approx k^\alpha$ means that the ratio between the
two quantities is bounded between two constants independent of $k$
(except for $w(0)$ on which we make no requirements).

The proof of the theorem consists of a martingale argument, which is a
modification of a similar martingale argument used in
\cite{davis1990reinforced} for edge-reinforced random walk 
on $\mathbb Z$.
Another ingredient of the proof 
is the fact that for small $\alpha$, the random walk will not
visit the nearby sites too many times before moving to a new site (see Lemma \ref{range_increasing}). This fact was basicly
proved by Schapira via a kind of domino principle. This is the part of
the proof that only works for $0\le \alpha<1/2$.

Let us remark that Arvind Singh \cite{Arvind2014} arrived at a similar result
simultaneously, using a martingale argument similar in spirit but
different in some technical details.

\section{Proof}

Recall the definition of $Z_i(j)$. For all $i\in \N$, we define a
sequence of random variables $F_i: \mathbb Z\setminus\{0\} \to \mathbb R^+$
\begin{equation}  \label{def_harmonic}
F_i(v)=\begin{cases}  
\sum_{j=0}^{v-1} \dfrac{1}{w(Z_i (j))\cdot w(Z_i (j+1))} & \text{ if } v>0;  \\
\sum_{j=v}^{-1} \dfrac{1}{w(Z_i (j))\cdot w(Z_i (j+1))} & \text{ if }  v<0. \\
\end{cases}
\end{equation}
Note that $F_i(\cdot)$ depends on the history of the random walk up to
time $i$ and is $\mathcal F_i-$measurable where $\mathcal F_i$ is the
$\sigma$-field spanned by $X_1,\dotsc,X_i$. Then we
have the following lemma.

\begin{lemma}  \label{supermartingale}
Let $X_0=0$. Let $T=\min\{i>0:X_i=0\}$ i.e.\ the first time the
process returns to the origin. Then $\{F_{\min(T,i)}(X_{\min(T,i)}):i=1,2,\dotsc\}$ is a supermartingale.
\end{lemma}

\begin{proof} 
We think about moving from $F_i(X_i)$ to $F_{i+1}(X_{i+1})$ as
being composed of two steps: moving $X$ and updating the weights. 
We will prove that $F_i$ satisfies the following two properties:
\begin{enumerate}
\item \label{enu:harm}\emph{harmonicity:} for all $i\in \N$, with respect to the random walk's transition probability at time $i$, $F_i(v)$ 
is harmonic (in $v$) on $\mathbb Z\setminus\{0\}$. In other words, the
first step is a martingale.
\item \label{enu:monot}\emph{monotonicity:} for any fixed $v\in \mathbb Z\setminus\{0\}$, $F_i(v)$ is monotone decreasing in $i$.
\end{enumerate}
Let us prove (\ref{enu:harm}). We condition on $\mathcal F_i$, and denote $v=X_i$ for
brevity, and assume $v>0$ (the other case is similar). We get
\begin{align*}
\mathbb{E}(F_i(X_{i+1})&\,|\,\mathcal{F}_i) =\mathbb P(X_{i+1}=v+1) F_i(v+1) +\mathbb P(X_{i+1}=v-1) F_i(v-1)   \\
&= \frac{w(Z_i(v+1))}{w(Z_i(v-1))+w(Z_i(v+1))} \sum_{j=0}^{v} \frac{1}{w(Z_i (j))\cdot w(Z_i (j+1))}   \\
&\qquad +\frac{w(Z_i(v-1))}{w(Z_i(v-1))+w(Z_i(v+1))} \sum_{j=0}^{v-2} \frac{1}{w(Z_i (j))\cdot w(Z_i (j+1))}   \\
&= \sum_{j=0}^{v-2} \frac{1}{w(Z_i (j))\cdot w(Z_i (j+1))}
 +  \frac{w(Z_i(v+1))}{w(Z_i(v-1))+w(Z_i(v+1))}\;\cdot\\
&\qquad\cdot\; \left( \frac{1}{w(Z_i(v-1))\cdot w(Z_i (v))}+ \frac{1}{w(Z_i (v))\cdot w(Z_i (v+1))} \right) \\
&= \sum_{j=0}^{v-1} \frac{1}{w(Z_i (j))\cdot w(Z_i (j+1))} = F_i(v).
\end{align*}
Hence, we proved (\ref{enu:harm}).

(\ref{enu:monot}) follows from the fact that for fixed $j$, $Z_i(j)$, the random walk's local time is monotone increasing in time $i$.

Now by harmonicity and monotonicity of $F_i(v)$, one has
\[
\mathbb E (F_{i+1}(X_{i+1})\, |\, \mathcal F_i)\le 
\mathbb E(F_i(X_{i+1})\,|\,\mathcal F_i)
= F_i(X_i),
\]
so $F_i(X_{i})$ is a supermartingale.
\end{proof}

\begin{remark*}
Lemma \ref{supermartingale} holds more generally for any vertex
reinforced random walk on $\mathbb Z$ with increasing weight
sequence. In fact it holds for any self-interacting process where the
vertex weights are increasing, and $\mathbb Z$ may be replaced with
any tree (also remarked in \cite{amir2008rwce}).
\end{remark*}

To prove the theorem we need a second lemma. Let $T_n$ denote the hitting time of a vertex $n \in \mathbb Z$. Then,
\begin{lemma} \label{range_increasing}
Almost surely, $I:=\liminf_{n\to \infty} Z_{T_n}(n-1)< \infty$.
\end{lemma}

\begin{proof}
The claim is equivalent to showing
\begin{equation}  \label{Borel_c}
\lim_{k\to\infty}\mathbb P \left( \liminf_{n\to \infty} Z_{T_n}(n-1)>k \right) =0.
\end{equation}
Note that for any fixed $k$
\begin{align*}
\mathbb P \left( \liminf_{n\to \infty} Z_{T_n}(n-1)>k \right) &=  \mathbb P \left( \cup_{N\ge 0}\cap_{n\ge N} \{Z_{T_n}(n-1)>k\} \right)   \\
                                                     &=  \sup_{N\ge 0} \mathbb P \left( \cap_{n\ge N}\{ Z_{T_n}(n-1)>k\} \right)  \\
                                                     &\le  \sup_{N\ge 0} \mathbb P \left(  Z_{T_N}(N-1)>k \right).
\end{align*}
We now apply formula (4.3) in \cite{schapira20110}, which claims that
\begin{equation}  \label{schapira_result}
\sup_{N\ge 0} \mathbb P \left(  Z_{T_N}(N-1)>k \right)\le Ce^{-ck^c}  ,
\end{equation}
where $c$ and $C$ are some positive constants (possibly depending on
the weight $w$).
Hence, (\ref{Borel_c}) follows from (\ref{schapira_result}).  This concludes the proof of the lemma.
\end{proof}

By the same argument as the proof of Lemma \ref{range_increasing}, one
can prove the same behaviour in the negative direction i.e.\ 
$\liminf_{n\to -\infty} Z_{T_n}(n+1)< \infty$.

Finally, we also use the 0-1 law proved by Schapira, which is stated as follows.
\begin{lemma}{\cite{schapira20110}*{Theorem 1.1}} \label{Shapira0-1law}
Vertex-reinforced random walk on $\mathbb Z$ with weight $w(k)\approx k^\alpha$, $k\ge 1$ for some
$\alpha\in [0, 1/2)$, is either recurrent or transient.
\end{lemma}

\begin{proofof}{the theorem}
By Lemma \ref{Shapira0-1law}, we know that $X_n$ is either recurrent or transient. Now
suppose $X_n$ is transient, then $X_n$ will visit the origin just finitely many times almost surely. By Lemma \ref{supermartingale},
$F_i(X_{i})$ will be a supermartingale eventually. Since it is positive, it converges to a finite random variable almost surely.
On the other hand, by Lemma \ref{range_increasing} there will be infinitely many vertices $N$, such that the increment of $F_i(X_{i})$
at time $T_N$ is bounded from below by a positive random
variable. Indeed, the only update to $Z$ that happens at time $T_N$ is
the increasing of $Z(N)$ to 1, but $Z(N)$ does not appear in the sum
defining $F_{T_n-1}$. Hence
\[
F_{T_n}(X_{T_n})-F_{T_n-1}(X_{T_n-1})
=\frac{1}{w(Z_{T_N}(N-1))w(1)}
\ge\frac{1}{w(I)w(1)}>0
\]
(where $I$ is still $\liminf_{n\to \infty} Z_{T_n}(n-1)< \infty$). This contradicts the convergence of $F_i(X_{i})$. Therefore, we can
conclude the theorem.
\end{proofof}
\section{Acknowledgements}

We thank Ronen Eldan and Cyrille Lucas for many enlightening discussions. During the research, J.C.\ was a student at
the Weizmann Institute of Science, supported by the
Israel Science Foundation. G.K.\ partially supported by the Israel Science Foundation.

\bibliography{sublinear_VRRW_bib}

\begin{bibdiv}
\begin{biblist}

\bib{amir2008rwce}{unpublished}{
      author={Amir, Gideon},
      author={Benjamini, Itai},
      author={Gurel-Gurevich, Ori},
      author={Kozma, Gady},
       title={Random walk in changing environment},
        date={unpublished manuscript, circa 2006},
}

\bib{davis1990reinforced}{article}{
      author={Davis, Burgess},
       title={Reinforced random walk},
        date={1990},
     journal={Probab. Theory Related Fields},
      volume={84},
      number={2},
       pages={203\ndash 229},
        note={Available at:
  \href{http://link.springer.com/article/10.1007\%2FBF01197845}{\nolinkurl{spr%
inger.com}}},
}

\bib{pemantle1992vertex}{article}{
      author={Pemantle, Robin},
       title={Vertex-reinforced random walk},
        date={1992},
     journal={Probab. Theory Related Fields},
      volume={92},
      number={1},
       pages={117\ndash 136},
        note={Available at:
  \href{http://link.springer.com/article/10.1007\%2FBF01205239}
  {\nolinkurl{springer.com}},
  \href{http://www.math.upenn.edu/~pemantle/papers/vrrw1.pdf}
  {\nolinkurl{upenn.edu/~pemantle}}},
}

\bib{pemantle1999vertex}{article}{
      author={Pemantle, Robin},
      author={Volkov, Stanislav},
       title={Vertex-reinforced random walk on \textbf{Z} has finite range},
        date={1999},
     journal={Ann. Probab.},
      volume={27},
      number={3},
       pages={1368\ndash 1388},
        note={Available at:
  \href{http://projecteuclid.org/euclid.aop/1022677452}{\nolinkurl{projecteucl%
id.org}}},
}

\bib{schapira20110}{article}{
      author={Schapira, Bruno},
       title={A 0-1 law for vertex-reinforced random walks on $\mathbb{Z}$ with
  weight of order $k^\alpha$, $\alpha< 1/2$},
        date={2012},
     journal={Electron. Commun. Probab.},
      volume={17},
      number={22},
        note={Available at:
  \href{http://ecp.ejpecp.org/article/view/2084}{\nolinkurl{ejpecp.org}}},
}

\bib{Arvind2014}{unpublished}{
      author={Singh, Arvind},
       title={Recurrence for vertex-reinforced random walks on $\mathbb{Z}$
  with weak reinforcements},
        date={preprint, 2014},
}

\bib{tarres2004vertex}{article}{
      author={Tarr{\`e}s, Pierre},
       title={Vertex-reinforced random walk on $\mathbb{Z}$ eventually gets
  stuck on five points},
        date={2004},
     journal={Ann. Probab.},
      volume={32},
      number={3B},
       pages={2650\ndash 2701},
        note={Available at:
  \href{http://projecteuclid.org/euclid.aop/1091813627}{\nolinkurl{projecteucl%
id.org}}},
}

\bib{volkov2006phase}{article}{
      author={Volkov, Stanislav},
       title={Phase transition in vertex-reinforced random walks on
  $\mathbb{Z}$ with non-linear reinforcement},
        date={2006},
     journal={J. Theoretic. Probab.},
      volume={19},
      number={3},
       pages={691\ndash 700},
        note={Available from:
  \href{http://link.springer.com/article/10.1007\%2Fs10959-006-0033-2}
  {\nolinkurl{springer.com}},
  \href{http://www.maths.lth.se/matstat/staff/s.volkov/PAPERS/nlr2.pdf}
  {\nolinkurl{lth.se/.../s.volkov}}},
}

\end{biblist}
\end{bibdiv}

\end{document}